\documentclass[12pt]{article}
\usepackage[utf8]{inputenc}
\usepackage[T1,T2A]{fontenc}
\usepackage[russian,english]{babel}
\usepackage[a4paper,left=20mm,right=20mm, top=10mm,bottom=30mm,bindingoffset=0cm]{geometry}
\usepackage{amsmath}
\usepackage{amssymb}
\usepackage{amsthm}
\usepackage{indentfirst}
\usepackage{diagbox}
\usepackage{graphicx}
\usepackage{wrapfig}
\usepackage{pgfpages}

\usepackage[hidelinks]{hyperref}
\hypersetup{
    pdftitle={On the Growth of Denominators of Simultaneous Best Diophantine Approximations in a Norm Induced by an Inner Product},
    pdfauthor={L.M. Shatunov}
}

\usepackage{caption}
\usepackage{cancel}

\theoremstyle{definition}

\newtheorem{proposition}{Proposition}
\newtheorem{lemma}{Lemma}
\newtheorem{corollary}{Corollary}
\newtheorem*{corollary*}{Corollary}

\theoremstyle{plain}
\newtheorem{theorem}{Theorem}
\newtheorem*{theorem*}{Theorem}

\theoremstyle{remark}

\title{On the Growth of Denominators of Simultaneous Best Diophantine Approximations in a Norm Induced by an Inner Product}
\author{L.M. Shatunov}
\date{}

\begin{document}
\selectlanguage{english}

    \maketitle

    \begin{abstract}
        For $n$-dimensional simultaneous best Diophantine approximations in an arbitrary norm induced by an inner product, for $n\geq2$ we prove
        $q_{k+2^n}\geq q_k+\min\{q_{k+2^{n-1}},2q_{k+1}\}$.
        This yields
        $$
        \displaystyle g_n(\alpha):=\liminf_{m\to\infty}(q_m)^{1/m}\geq\varphi^{1/2^{n-1}}, \qquad \text{for } \ \varphi = \dfrac{1 + \sqrt{5}}{2}.
        $$ Consequently, \ensuremath{\displaystyle G(n):=\inf_{\alpha\in\mathbb R^n\setminus\mathbb Q^n}g_n(\alpha)\geq\varphi^{1/2^{n-1}}} and
        $\displaystyle \underline{\mathcal D}_n(\alpha)\leq\left\lfloor 2^{n-1}\frac{\log2}{\log\varphi}\right\rfloor+1$, where \(\underline{\mathcal D}_n(\alpha)\) is a quantity related to multidimensional generalizations of the three-distance theorem.
        In particular,$$g_2(\alpha)\geq\sqrt\varphi, \qquad g_3(\alpha)\geq\sqrt[4]{\varphi}, \qquad \underline{\mathcal D}_2(\alpha)\leq3, \qquad \underline{\mathcal D}_3(\alpha)\leq6.$$
    \end{abstract}

    \bigskip

    This paper is devoted to the exponential growth rates of the denominators of simultaneous best approximations in a space equipped with a norm induced by an inner product. The proof is based on additive combinatorics over $\mathbb F_2$ and uses, in particular, Kneser's theorem~\cite{kneser}.

    The paper is organized as follows. Section~1 gives the necessary definitions and formulates the problems concerning the growth of simultaneous best Diophantine approximations and multidimensional generalizations of the three-distance theorem. Section~2 states the main $n$-dimensional denominator inequality and its consequences for the exponential growth rate and for $\underline{\mathcal D}_n(\alpha)$. Section~3 gives an outline of the proof and introduces the abbreviated notation used in it. Section~4 proves the stated denominator inequality.

    \section{Introduction}

    \subsection{Simultaneous best Diophantine approximations}

    In $\mathbb R^n$, consider an arbitrary norm $|\cdot|$, induced by an inner product $\langle\cdot,\cdot\rangle$, that is,
    $
        \ensuremath{|x|=\sqrt{\langle x,x\rangle}}.
    $
    After a linear change of coordinates, this setting is equivalent to the setting with an arbitrary full-rank lattice in an $n$-dimensional real inner-product space used by Shulga~\cite{shulga}.

    For fixed $\alpha\in\mathbb R^n\setminus\mathbb Q^n$, set
    $$
        \delta_q=\min_{\mathbf p\in\mathbb Z^n}\ensuremath{|q\alpha-\mathbf p|}.
    $$
    The denominators of best approximations are defined to be the increasing sequence $1=q_1<q_2< q_3 <\ldots$ consisting of all positive integers $q$ such that $\delta_t>\delta_q$ for every $0<t<q$.

    In the present paper we study the growth and structure of denominators of simultaneous best approximations in arbitrary dimension. Recent work of Shutov~\cite{shutov} and Shulga~\cite{shulga} revealed the connection between denominator-growth inequalities and multidimensional nearest-neighbor versions of the three-distance theorem, building on Chevallier's results~\cite{chevallier_dist} and the geometric approach of Haynes and Marklof~\cite{haynes_marklof}. We obtain a new dimension-uniform denominator inequality and derive consequences for the associated growth exponent and for the asymptotic number of nearest-neighbor distances.

    As the main characteristic of the growth of this sequence, we consider
    $$
        g_n(\alpha)=\liminf_{m\to\infty}\left(q_m\right)^{1/m},
        \qquad
        G(n)=\inf_{\alpha\in\mathbb R^n\setminus\mathbb Q^n}g_n(\alpha).
    $$
    For the Euclidean norm, we denote the corresponding quantity by \ensuremath{G_e(n)}.
    Throughout the paper, $\theta_s$ denotes the largest root of the equation $\theta^s = \theta + 1$.

    In the one-dimensional case, the denominators of best approximations coincide with the denominators of the convergents of the continued fraction, and hence $q_{m+2}=a_{m+2}q_{m+1}+q_m \Rightarrow q_{m+2} \geq q_{m+1} + q_m$.
    Thus $g_1(\alpha)\geq\varphi$, where $\varphi=\dfrac{1+\sqrt5}{2} = 1.6180339^+$. But $g_1(\varphi) = \varphi$, and therefore $G(1)=\varphi$~\cite{khinchin}.


    The study of the growth of denominators was initiated by Lagarias~\cite{lagarias}. For an arbitrary norm in $\mathbb R^n$, he proved
    $$
        q_{m+2^{n+1}}\geq2q_{m+1}+q_m.
    $$
    It follows that $G(n)\geq\lambda_n^{\mathrm L}$, where $\lambda_n^{\mathrm L}>1$ is the root of
    $
        x^{2^{n+1}}=2x+1.
    $
    In particular, this general result gives 
    $$G(2)\geq1.1620043^+, \qquad G(3)\geq1.0743174^+.
    $$

    In $\mathbb{R}^n$ equipped with an arbitrary norm \ensuremath{\|\cdot\|}, there is also a general bound in terms of the kissing number \ensuremath{K(n,\|\cdot\|)}:
    $$
        q_{m+K(n,\|\cdot\|)}\geq q_{m+1}+q_m.
    $$
    A proof of this statement is given, for example, in the work of Romanov~\cite{romanov}. For the Euclidean norm, this gives \ensuremath{G_e(2)\geq\theta_6=1.1347241^+} and \ensuremath{G_e(3)\geq\theta_{12}=1.0621691^+}; however, the bound for \ensuremath{G_e(2)} is trivially improved to $\theta_5=1.1673039^+$, as also noted by Romanov. Shulga~\cite{shulga} also discusses the corresponding refinement in terms of the strict Euclidean kissing number $\sigma_n^{>}$, which yields the recurrence $q_{m+\sigma_n^{>}}\geq q_{m+1}+q_m$.

    For the Euclidean norm, Romanov~\cite{romanov} proved 
    $$q_{m+4}\geq q_{m+1}+q_m. \qquad \text{Hence } \ \ensuremath{G_e(2)\geq\theta_{4}=1.220744^+}.
    $$

    Ermakov strengthened this result to \ensuremath{G_e(2)\geq1.228043}~\cite{ermakov}, proving that for every positive integer $m$ one has
    $$
        q_{m+3}+q_{m+2}\geq2q_{m+1}+q_m
        \quad\text{or}\quad
        q_{m+4}\geq q_{m+2}+q_m,
    $$
    and, among any two consecutive values of $m$, the first inequality must hold for at least one of them. A finite computer search over the admissible transitions then gives the constant $1.228043$.


    More details and other results can be found in surveys of Moshchevitin~\cite{moshchevitin_survey} and Chevallier~\cite{chevallier_survey}.

    Very recently Shulga obtained a result in an equivalent general setting~\cite{shulga}. He proved that either $q_{m+2^n}\geq2q_{m+1}$, or the set $\{1,\ldots,2^n\}$ can be partitioned into pairs $\{j,k\}$, $j<k$, for which $q_{m+k}=q_m+q_{m+j}$. In particular,
    $$
        q_{m+2^n}\geq
        \min\{2q_{m+1},\,q_m+q_{m+2^{n-1}}\}
    $$
    and, for $n\geq2$, it follows that
    $$
        G(n)\geq2^{1/(2^n-1)}.
    $$
    Thus $G(2)\geq2^{1/3}=1.259921^+$, and also $G(3)\geq2^{1/7}=1.1040895^+$.

    In the present paper, we strengthen this denominator inequality to
    $$
        q_{m+2^n}\geq q_m+\min\{q_{m+2^{n-1}},2q_{m+1}\},
    $$
    $$
        \text{which yields } \quad G(n)\geq\varphi^{1/2^{n-1}},\quad n\geq2.
    $$
    In particular,
    $$
        G(2)\geq\sqrt\varphi=1.2720196^+,
        \qquad
        G(3)\geq\sqrt[4]{\varphi}=1.127838^+.
    $$

    In the Euclidean case, our bound is stronger than the bound obtained from the strict kissing number for $2\leq n\leq8$; the crossover already occurs in dimension $9$, where the bound $\sigma_9^{>}\leq363$ quoted in~\cite{shulga} gives \ensuremath{G_e(9)\geq\theta_{363}>\varphi^{1/256}}. For large dimensions, the kissing-number method is asymptotically stronger.

    \subsection{The three distance theorem and its generalizations}

    A related problem concerns distances between points of a Kronecker sequence on the torus. For $x,y\in\mathbb T^n=\mathbb R^n/\mathbb Z^n$ and fixed $\alpha\in\mathbb R^n\setminus\mathbb Q^n$ and $N\in\mathbb N \setminus \{0\}$, set
    $$
        \rho(x,y)=\min_{z\in\mathbb Z^n}\ensuremath{|x-y-z|},
        \qquad
        D_q(N)=\min_{\substack{0\leq t\leq N\\t\ne q}}
        \rho(q\alpha,t\alpha),
    $$
    $$
        \mathcal D_n(\alpha,N)=\#\{D_q(N)\colon0\leq q\leq N\}.
    $$
    Thus, $\mathcal D_n(\alpha,N)$ is the number of distinct distances from the points $0,\alpha,\ldots,N\alpha$ to their nearest neighbors. We also set
    $$
        \mathcal{D}_n(\alpha)=\sup_N \mathcal{D}_n(\alpha,N),
        \qquad
        \underline{\mathcal D}_n(\alpha)=\liminf_{N\to\infty}\mathcal D_n(\alpha,N).
    $$

    The problem is to determine the exact upper bound for $\mathcal{D}_n(\alpha)$ valid for all $\alpha$.

    The classical three distance theorem, originating in a conjecture of Steinhaus, was proved in the late 1950s; one of the first proofs is due to Sós~\cite{sos}. Namely, the points $0,\alpha,\ldots,N\alpha$ on the circle partition it into intervals of at most three distinct lengths~\cite{sos}.

    In dimension two, Chevallier established an exact relation between the number of nearest-neighbor distances and the denominators of best approximations~\cite{chevallier_dist}. If $q_m\leq N<q_{m+1}$ and $2q_l\leq N<2q_{l+1}$, then the number of distinct distances is equal to $m-l$ or $m-l+1$, according as whether $2q_{l+1}=N+1$ holds.

    For the Euclidean metric, Haynes and Marklof explicitly proved the two-dimensional five distance theorem~\cite{haynes_marklof}: for every $\alpha$ and every $N$, one has $\mathcal D_2(\alpha,N)\leq5$; they also proved that this bound is sharp. The upper bound $5$ itself already follows by combining Chevallier's lemma with Romanov's inequality for the denominators of simultaneous approximations; this connection was explicitly pointed out by Shutov~\cite{shutov}. In dimension $3$, Haynes and Marklof obtained the bound $13$ and conjectured that the optimal constant is $9$.

    Shutov showed how, using Chevallier's lemma, inequalities for the denominators of best approximations can be directly translated into uniform bounds for the number of distances~\cite{shutov}. In particular, if for some $T$ one has $q_{m+T}\geq2q_m$ for all $m$, then $\mathcal D_n(\alpha)\leq T+1$. In addition, Shutov studied the asymptotic quantity $\underline{\mathcal D}_n(\alpha)$ and showed that if $q_{m+T}\geq2q_m$ holds for infinitely many $m$, then $\underline{\mathcal D}_n(\alpha)\leq T$.

    Shulga~\cite{shulga} proved a general result which gives the bound
    $$
        \mathcal D_n(\alpha,N)\leq2^n+1.
    $$
    In particular, for $n=3$ he proved the nine distance theorem, and an example of Dettmann~\cite{dettmann} shows that the bound $9$ is sharp. The statements proved by Shulga also imply the growth estimate for denominators of best approximations
    $$
        q_{m+2^n} \geq q_m + q_{m + 1}>2q_m.
    $$

    Our denominator-growth estimate implies the following bound
    $$
        \underline{\mathcal D}_n(\alpha)
        \leq
        \left\lfloor 2^{n-1}\frac{\log2}{\log\varphi}\right\rfloor+1,
        \qquad n\geq2.
    $$
    $$
        \text{In particular,} \qquad
        \underline{\mathcal D}_2(\alpha)\leq3,
        \qquad
        \underline{\mathcal D}_3(\alpha)\leq6.
    $$

    \section{Main results}

    In this section we formulate the main denominator inequalities and immediately derive their consequences for the exponential growth exponent and for the asymptotic nearest-neighbor distance quantity $\underline{\mathcal D}_n(\alpha)$.

    \begin{theorem}\label{theorem:Rn}
        For every $n\geq2$ and every $k\geq1$,
        $$
            q_{k+2^n}\geq q_k+\min\{q_{k+2^{n-1}},2q_{k+1}\}.
        $$
    \end{theorem}



    \begin{corollary}\label{cor:Rn-growth}
        For every $n\geq2$ and every $\alpha\in\mathbb R^n\setminus\mathbb Q^n$ one has
        $$
            g_n(\alpha)\geq\varphi^{1/2^{n-1}}.
        $$
        Consequently,
        $
            G(n)\geq\varphi^{1/2^{n-1}}.
        $
    \end{corollary}

    \begin{proof}
        Put $M=2^n$ and $\lambda=\varphi^{1/2^{n-1}}$. Then
        $$
            \lambda^{M/2}=\varphi,
            \qquad
            \lambda^M=\varphi^2=1+\varphi,
            \qquad
            \varphi<2\lambda.
        $$
        Choose $C>0$ such that $q_j\geq C\lambda^j$ for $1\leq j\leq M$. We prove by induction that $q_l\geq C\lambda^l$.

        If the assertion has already been proved up to index $l+M-1$, then Theorem~\ref{theorem:Rn} gives
        $$
            q_{l+M}\geq q_l+\min\{q_{l+M/2},2q_{l+1}\}
            \geq C\lambda^l\bigl(1+\min\{\lambda^{M/2},2\lambda\}\bigr)
            =C\lambda^{l+M}.
        $$
    \end{proof}

    \begin{corollary}\label{cor:Rn-distances}
        For every $n\geq2$ and every $\alpha\in\mathbb R^n\setminus\mathbb Q^n$ ,
        $$
            \underline{\mathcal D}_n(\alpha)
            \leq
            \left\lfloor
                2^{n-1}\frac{\log2}{\log\varphi}
            \right\rfloor+1.
        $$
        In particular,
        $
            \underline{\mathcal D}_2(\alpha)\leq3,
            \underline{\mathcal D}_3(\alpha)\leq6.
        $
    \end{corollary}

    \begin{proof}
        Put
        $$
            T_n=\left\lfloor
                2^{n-1}\frac{\log2}{\log\varphi}
            \right\rfloor+1.
        $$
        By Corollary~\ref{cor:Rn-growth},
        $$
            g_n(\alpha)
            \geq\varphi^{1/2^{n-1}}
            >2^{1/T_n}.
        $$
        Therefore the inequality $q_{m+T_n}\geq2q_m$ holds for infinitely many $m$. Indeed, otherwise $q_{m+T_n}<2q_m$ for every sufficiently large $m$, and iteration along each residue class modulo $T_n$ would imply
        $$
            \limsup_{m\to\infty}q_m^{1/m}\leq2^{1/T_n},
        $$
        a contradiction. Applying Lemma~3 of~\cite{shutov} with $T=T_n$, we obtain the required estimate.
    \end{proof}

    \section{Outline of the proof}

    We continue to work with the norm introduced above and the notation $\delta_q,q_m$. For each $q$, choose and fix one of the vectors $\mathbf p(q)$ at which the minimum in the definition of $\delta_q$ is attained. All subsequent arguments hold for any such fixed choice of $\mathbf p(q)$.

    Set $\mathbf r(q)=q\alpha-\mathbf p(q)$ and $L(q)=(q,\mathbf p(q))^T$. Extend the definitions by
    $\delta_{-q}=\delta_q$, $\mathbf p(-q)=-\mathbf p(q)$, $\mathbf r(-q)=-\mathbf r(q)$, and $L(-q)=-L(q)$. In addition,
    $\delta_0=\mathbf p(0)=\mathbf r(0)=L(0)=0$.

    Choose an arbitrary $k\in\mathbb N\setminus\{0\}$ and fix it. Put
    $$
        M=2^n,\qquad s=2^{n-1}+1,
    $$
    $$
        \text{and set} \qquad
        Q_i=q_{k+i-1},\qquad
        A_i=\mathbf r(Q_i),\qquad
        L_i=L(Q_i),\qquad
        R=\ensuremath{|A_1|}.
    $$
    For an integer vector $X$, let $\overline X$ denote its residue class modulo~$2$.

    Assume that
    $$
        Q_{M+1}<Q_1+Q_s
        \qquad\text{and}\qquad
        Q_{M+1}<Q_1+2Q_2.
    $$
    The proof of Theorem~\ref{theorem:Rn} consists of the following steps.
    \begin{enumerate}
        \item The classes $\overline L_s,\ldots,\overline L_{M+1}$ are pairwise distinct.
        \item Among the classes $\overline L_i+\overline L_j$, $s\leq i<j\leq M+1$, there are at least $M-1$ distinct ones.
        \item The set
        $$
            P=\{\overline L_i+\overline L_j:s\leq i<j\leq M+1\}
        $$
        is disjoint from $\{\overline L_1,\ldots,\overline L_{M+1}\}$.
        \item By the pigeonhole principle there exist $1 \leq i<j \leq M + 1$ such that $\overline L_i=\overline L_j$.
        \item The equality $\overline L_i=\overline L_j$ implies
        $$
            Q_j\geq Q_i+2Q_{i+1},
        $$
        and therefore
        $
            Q_{M+1}\geq Q_1+2Q_2,
        $
        a contradiction.
    \end{enumerate}

    \section{Proof of the theorem}

    Throughout this section, $n\geq2$, and we use the notation introduced in Section~3.



    \begin{lemma}
    \label{gcd L = 1}
        For every $m$, the vector $L(q_m)$ is primitive, that is, the gcd of its coordinates is equal to $1$.
    \end{lemma}

    \begin{proof}
        Suppose that there is an integer $d > 1$ such that $L(q_m) = d\begin{pmatrix} q' \\ \mathbf{p}' \end{pmatrix}$.

        Then $q' < q_m \Rightarrow \delta_{q'} > \delta_{q_m}$. However,
        $$
        \delta_{q'} = \min\limits_{t \in \mathbb{Z}^n}|q'\alpha - t| \leq \left| q'\alpha - \mathbf{p}' \right| = \dfrac{\delta_{q_m}}{d} < \delta_{q_m},
        $$
        a contradiction.
    \end{proof}

    Now we formulate our main tool.

    \begin{proposition}\label{prop:record-property}
        Let $l\geq1$, and let $T\in\mathbb Z\setminus\{0\}$ and
        $\mathbf P\in\mathbb Z^n$. Then
        \begin{itemize}
            \item if $|T|<Q_l \ \ \ $, then
            $
                |T\alpha-\mathbf P|>|A_l|;
            $
            \item if $|T|<Q_{l+1}$, then
            $
                |T\alpha-\mathbf P|\geq|A_l|.
            $
        \end{itemize}
    \end{proposition}

\begin{proof}
    Since $\mathbf P\in\mathbb Z^n$,
    $$
        |T\alpha-\mathbf P|\geq\delta_{|T|}.
    $$
    If $0<|T|<Q_l$, then by the definition of the best denominator
    $Q_l$,
    $$
        \delta_{|T|}>\delta_{Q_l}=|A_l|.
    $$
    If $0<|T|<Q_{l+1}$, then by the definition of the next best
    denominator $Q_{l+1}$,
    $$
        \delta_{|T|}\geq\delta_{Q_l}=|A_l|.
    $$
\end{proof}

    \begin{lemma}\label{rn:four-squares}
        For any $a,b,c\in\mathbb R^n$,
        $$
        |a+b-c|^2+|a-b+c|^2+|-a+b+c|^2+|a+b+c|^2
        =4(|a|^2+|b|^2+|c|^2).
        $$
    \end{lemma}

    \begin{proof}
        After expansion, each mixed inner product occurs twice with each sign.
    \end{proof}

    \begin{proposition}\label{rn:no-triple}
        If $Q_{M+1}<Q_1+Q_s$, then there do not exist $1\leq a<b<c\leq M+1$ with $b\geq s$ such that
        $$
            \overline L_a+\overline L_b+\overline L_c=0.
        $$
    \end{proposition}

    \begin{proof}
        Assume the contrary. Set
        $$
            U=A_a+A_b-A_c,\quad V=A_a-A_b+A_c,\quad
            W=-A_a+A_b+A_c,\quad Z=A_a+A_b+A_c.
        $$
        Since $\overline L_a+\overline L_b+\overline L_c=0$, the following vectors are integral
        $$
            \frac{L_a+L_b-L_c}{2},\quad
            \frac{L_a-L_b+L_c}{2},\quad
            \frac{-L_a+L_b+L_c}{2},\quad
            \frac{L_a+L_b+L_c}{2}.
        $$
        From
        $
            Q_a+Q_b-Q_c\geq Q_1+Q_s-Q_{M+1}>0
        $
        and $Q_b<Q_c$, we obtain
        $$
            0<\frac{Q_a+Q_b-Q_c}{2}<Q_a.
        $$
        Applying Proposition~\ref{prop:record-property} with
        $\displaystyle
            T=\frac{Q_a+Q_b-Q_c}{2}$, $\displaystyle 
            \mathbf P=\frac{\mathbf p(Q_a)+\mathbf p(Q_b)-\mathbf p(Q_c)}{2},
        $
        and $l=a$, we obtain
        $
            |U|>2|A_a|.
        $

\medskip

        Also,
        $
            Q_c-Q_b\leq Q_{M+1}-Q_s<Q_1\leq Q_a,
        $
        so
        $$
            0<\frac{Q_a-Q_b+Q_c}{2}<Q_a, \qquad \text{and similarly } \ |V|>2|A_a|.
        $$
        By Lemma~\ref{rn:four-squares},
        $$
        |W|^2+|Z|^2
        =4\bigl(|A_a|^2+|A_b|^2+|A_c|^2\bigr)-|U|^2-|V|^2
        <4\bigl(|A_b|^2+|A_c|^2-|A_a|^2\bigr)
        <4|A_c|^2.
        $$
        Hence $|W|<2|A_c|$. But
        $$
            0<\frac{-Q_a+Q_b+Q_c}{2}<Q_c,
        $$
        while the corresponding approximation vector is $W/2$, contradicting the definition of the best denominator $Q_c$.
    \end{proof}

    \begin{lemma}\label{rn:pair-sums}
        Let $A$ be a set of $2^{n-1}+1$ distinct elements of an $\mathbb F_2$-vector space. Then among the sums $x+y$ with $x,y\in A$, $x\ne y$, there are at least $2^n-1$ distinct elements.
    \end{lemma}

    \begin{proof}
        Let $H$ be the stabilizer of $A+A$ and put $h=|H|$. By Kneser's theorem~\cite{kneser},
        $$
            |A+A|\geq2|A+H|-|H|.
        $$
        Put $M=2^n$. Since $h$ is a power of $2$, either $h\geq M$ or $h\leq M/2$. In the first case $|A+A|\geq M$. In the second case, since $h$ divides $M/2$ and $|A|=M/2+1$, the set $A+H$ contains at least $M/(2h)+1$ cosets of $H$. Hence
        $$
            |A+H|\geq M/2+h,
        $$
        and therefore
        $$
            |A+A|\geq M+h\geq M.
        $$
        In characteristic $2$, the zero element belongs to $A+A$, and every nonzero element of $A+A$ is a sum of two distinct elements of $A$. Thus there are at least $M-1=2^n-1$ distinct sums with $x\ne y$.
    \end{proof}

    \begin{proposition}\label{rn:distinct-tail}
        If $Q_{M+1}<Q_1+Q_s$, then the classes
        $$
            \overline L_s,\ldots,\overline L_{M+1}\in\mathbb F_2^{n+1}
        $$
        are pairwise distinct.
    \end{proposition}

    \begin{proof}
        Let $s\leq i<j\leq M+1$ and suppose that $\overline L_i=\overline L_j$. Then $(L_j-L_i)/2\in\mathbb Z^{n+1}$, and
        $$
            0<\frac{Q_j-Q_i}{2}
            \leq\frac{Q_{M+1}-Q_s}{2}
            <\frac{Q_1}{2}<Q_i,
        $$
        $$
            \left|\frac{A_j-A_i}{2}\right|
            \leq\frac{|A_i|+|A_j|}{2}<|A_i|.
        $$
        Applying Proposition~\ref{prop:record-property} we obtain
        $$
            \left|\frac{A_j-A_i}{2}\right|>|A_i|,
        $$
        a contradiction.
    \end{proof}

    \begin{proposition}\label{rn:P-disjoint}
        If $Q_{M+1}<Q_1+Q_s$ and
        $$
            P=\{\overline L_i+\overline L_j:s\leq i<j\leq M+1\},
        $$
        $$
            \text{then} \quad
            P\cap\{\overline L_1,\ldots,\overline L_{M+1}\}=\varnothing.
        $$
    \end{proposition}

    \begin{proof}
        Suppose that $\overline L_i+\overline L_j=\overline L_t$ for some $s\leq i<j\leq M+1$. By Lemma~\ref{gcd L = 1}, the classes $\overline L_i$ and $\overline L_j$ are nonzero, so $t$ is distinct from $i,j$. Ordering the three indices, we obtain
        $$
            \overline L_a+\overline L_b+\overline L_c=0
        $$
        with $b\geq s$, contradicting Proposition~\ref{rn:no-triple}.
    \end{proof}

    \begin{proposition}\label{rn:equal-classes-gap}
        If $1\leq i<j\leq M+1$ and $\overline L_i=\overline L_j$, then
        $$
            Q_j\geq Q_i+2Q_{i+1}.
        $$
    \end{proposition}

    \begin{proof}
        By assumption, $(L_j-L_i)/2\in\mathbb Z^{n+1}$. Also,
        $$
            \left|\frac{A_j-A_i}{2}\right|
            \leq\frac{|A_i|+|A_j|}{2}<|A_i|.
        $$
        If
        $
            \dfrac{Q_j-Q_i}{2}<Q_{i+1},
        $
        then Proposition~\ref{prop:record-property}, gives
        $
            \left|\dfrac{A_j-A_i}{2}\right|\geq|A_i|,
        $
        contrary to the strict inequality above. Therefore
        $
            Q_j-Q_i\geq2Q_{i+1}.
        $
    \end{proof}
    
    \textbf{\hspace{-24pt} Proof of Theorem~\ref{theorem:Rn}.}
    Assume the contrary. Then
    $$
        Q_{M+1}<Q_1+Q_s
        \qquad\text{and}\qquad
        Q_{M+1}<Q_1+2Q_2.
    $$
    By Proposition~\ref{rn:distinct-tail}, the $M/2+1$ classes $\overline L_s,\ldots,\overline L_{M+1}$ are distinct. Hence, by Lemma~\ref{rn:pair-sums}, the set $P$ from Proposition~\ref{rn:P-disjoint} contains at least $M-1$ distinct nonzero classes. By Proposition~\ref{rn:P-disjoint}, none of them coincides with any of $\overline L_1,\ldots,\overline L_{M+1}$.

    By Lemma~\ref{gcd L = 1}, all the classes $\overline L_1,\ldots,\overline L_{M+1}$ are nonzero. There are exactly
    $$
        2^{n+1}-1=2M-1
    $$
    nonzero elements in $\mathbb F_2^{n+1}$. Therefore the $M+1$ classes $\overline L_1,\ldots,\overline L_{M+1}$ take at most
    $$
        (2M-1)-(M-1)=M
    $$
    distinct values. Hence $\overline L_i=\overline L_j$ for some $i<j$. By Proposition~\ref{rn:equal-classes-gap},
    $$
        Q_{M+1}\geq Q_j\geq Q_i+2Q_{i+1}\geq Q_1+2Q_2,
    $$
    contradicting the assumption. \qed

\end{document}